\documentclass[12pt]{article}
\usepackage{amsmath, amsthm, amssymb, fullpage, amscd}
\usepackage[english]{babel}
\usepackage[utf8]{inputenc}
\usepackage{eucal}
\usepackage{thmtools}
\usepackage{float}
\usepackage{thm-restate}
\usepackage{tikz-cd, pgf, tikz}
\usepackage{geometry}
\usepackage{titlesec}
\usepackage{xcolor}
\usepackage{mathrsfs}
\usepackage{mathtools}
\usepackage[nottoc]{tocbibind}
\usepackage[noadjust]{cite}
\usepackage{graphicx}
\usepackage{caption}
\usepackage{subcaption}
\usepackage{bm, bbm}
\usepackage{mathrsfs}
\usepackage{array}
\usepackage{latexsym}
\usepackage{stmaryrd}
\usepackage{commath,enumitem,chngcntr}
\usepackage{indentfirst}
\usepackage[titletoc,toc,title]{appendix}
\usepackage{pst-node}
\usepackage{yhmath}
\usepackage{appendix}
\usepackage{algorithm}
\usepackage[noend]{algpseudocode}
\makeatletter
\renewcommand{\theHALG@line}{\thealgorithm.\arabic{ALG@line}}
\makeatother

\usepackage{float}
 \usepackage{overpic} 
\usepackage{faktor}

\usepackage[colorlinks, bookmarks=true, citecolor=blue, linkcolor=blue, urlcolor=blue]{hyperref}
\usepackage{cleveref}

\setlist{itemsep=3pt plus 1pt minus 1pt}

\usetikzlibrary{arrows,calc,positioning}

\newcommand{\dmo}{\DeclareMathOperator}
\dmo{\Id}{Id}
\dmo{\pt}{pt}
\dmo{\ab}{ab}
\dmo{\GL}{GL}
\dmo{\SO}{SO}
\dmo{\SU}{SU}
\dmo{\SL}{SL}
\dmo{\PSL}{PSL}
\dmo{\im}{Im}
\dmo{\re}{Re}
\dmo{\Ad}{Ad}
\dmo{\ad}{ad}
\dmo{\tr}{tr}
\dmo{\Log}{Log}
\dmo{\Homeo}{Homeo}
\dmo{\Vol}{Vol}
\dmo{\supp}{supp}
\dmo{\Isom}{Isom}
\dmo{\Hom}{Hom}
\dmo{\Ext}{Ext}
\dmo{\Tor}{Tor}
\dmo{\ind}{ind}
\dmo{\PP}{\mathbf{P}}

\newcommand{\C}{\mathbb C}

\newcommand{\D}{\mathbb D}

\newcommand{\Arg}{\mathrm {Arg}}
\renewcommand{\tr}{\mathrm {tr}}

\newtheorem{theorem}{Theorem}[section]
\newtheorem{corollary}[theorem]{Corollary}

\newtheorem{lemma}[theorem]{Lemma}

\theoremstyle{definition}

\theoremstyle{remark}

\def\noteson{\gdef\yanwen##1{\noindent{\color{blue}[Yanwen: ##1]}}
\gdef\yuan##1{\noindent{\color{violet}[Yuan: ##1]}}
\gdef\todo##1{\noindent{\color{red}[todo: ##1]}}}

\noteson

\usepackage{tikz-cd}
\usetikzlibrary{cd}

\makeatletter
\let\c@equation\c@theorem
\makeatother
\numberwithin{equation}{section}

\usepackage{lipsum}
\let\OLDthebibliography\thebibliography
\renewcommand\thebibliography[1]{
	\OLDthebibliography{#1}
	\setlength{\parskip}{2pt}
	\setlength{\itemsep}{0pt plus 0.3ex}
}

\usepackage{authblk}
\usepackage{hyperref}
\title{Morphing Graphs on Hyperbolic Surfaces}
\author{Yanwen Luo\footnote{University of Science and Technology Beijing, Beijing, 100083, P.R. China}, Yuan Luo\footnote{Department of Mathematics, University of California Davis, Davis, 95616, U.S}}
\date{July 2026}

\begin{document}

\maketitle
\begin{abstract}


We propose the first algorithm to morph geometric graphs on hyperbolic surfaces. It is based on a generalization of Tutte's spring embedding theorem on essentially $3$-vertex-connected graphs. We describe the algorithms in detail and show experiments with triangulations and graphs on a hyperbolic surface of genus two, the Bolza surface, and a hyperbolic surface of genus three, the Klein quartic.

\end{abstract}

\section{Introduction}

Constructing continuous deformations between two geometric shapes is a central problem in computational geometry. The goal of this paper is to introduce an algorithm with a solid mathematical foundation to compute a continuous deformation, or a morphing,  between two geometric graphs embedded in hyperbolic surfaces.

The graph morphing problem is concerned with continuously transforming one geometric embedding of a graph into another while preserving key properties such as planarity, non-self-intersection, and other geometric constraints. Starting from Cairns \cite{Cairns,Cairns2}, this problem in the Euclidean plane has been studied extensively for planar graphs \cite{BS,Ho1,Ho2,Bloch,BCH, Luo, Thomassen} from a topological perspective, especially on the connectedness of spaces of geometric embeddings. The theoretical results lead to various efficient algorithms \cite{planar, CGT, EKC, FG, GS}.
Here, a geometric embedding means that every edge of the graph is embedded as a straight line segment in the Euclidean plane, also called a straight-line embedding or a drawing. It produces a continuous family of straight-line embeddings parameterized by time, starting with an initial embedding and ending at another, with applications in surface parametrization and computer graphics. 

It is a natural question to extend the graph morphing problem on more complicated surfaces with non-trivial topology and non-Euclidean geometry. 
By the topological classification theorem of surfaces,  connected, closed, orientable surfaces are classified up to homeomorphism by their genus $g$. Geometrically, the uniformization theorem implies that such a surface can be equipped with metrics of constant curvature, whose sign is completely determined by its topology. Specifically, the unit $2$-sphere ($g=0$) in the Euclidean space has constant curvature $+1$, flat tori ($g = 1$)  have constant curvature $0$, and hyperbolic surfaces ($g\geq 2$) have constant curvature $-1$. A geometric embedding of a graph, or a geodesic embedding, on these surfaces refers to an embedding of the graph such that every edge is embedded as a geodesic segment with respect to metrics of constant curvature. It is determined by the position of each vertex of the graph and the homotopy class of each edge. A morphing between two geodesic embeddings on them is an isotopy, namely a continuous family of geodesic embeddings, between the given two embeddings.

\paragraph{Related work.} Compared with the massive literature on morphing graphs in the plane, the graph morphing problem on surfaces remains an active research subject with many open problems. It is curious that effective algorithms to morph graphs on the unit $2$-sphere remain unknown, given the fact that it is the most natural surface to consider, and morphing on the $2$-sphere is topologically equivalent but geometrically distinct from morphing planar graphs. Partial results provide morphing algorithms for special classes of geodesic embeddings along longitude \cite{AH, EH, KL}, or with angle structures \cite{LWZ2}. One of the main difficulties stems from the complication of positive curvature, such as large geodesic triangles or non-uniqueness of geodesic segments between two points on the $2$-sphere.

Graph morphing problem on flat tori has been solved recently in the work of \cite{CELP} and \cite{EL2021} from two distinct perspectives. The first follows Cairns' approach via direct deformations of vertices and collapses of edges, taking advantage of the combinatorial properties of toroidal triangulations. The second follows Floater-Gotsman's approach in \cite{FG} with generalizations of the classical Tutte's spring theorem \cite{Tutte} by lifting the geodesic embeddings of a triangulation to spaces of weights in Tutte's theorem as parameter spaces. Both algorithms have been shown to be efficient in terms of computational complexity.

\paragraph{Our work.} We describe an algorithm to morph two geodesic embeddings on hyperbolic surfaces following Floater-Gotsman's approach. It is computed in the universal covering of hyperbolic surfaces, the Poincar\'e disk $\D$. A generalization of Tutte's embedding theorem of triangulations on hyperbolic surfaces has been established in \cite{LWZ}. We further extend it to essentially $3$-connected graphs in the Appendix, which allows morphing between two convex geodesic embeddings of a cellular decomposition. It provides the theoretical foundation to show the existence of intermediate embeddings in the morphing. To our knowledge, this is a first algorithm to morph graphs on surfaces of genus at least two.


This paper is organized as follows. 
In Section 2, we collect mathematical backgrounds to describe the computation with the hyperbolic disk, Möbius transformations, and basic concepts related to the generalization of Tutte's embedding theorem on hyperbolic surfaces and mean value coordinates. We introduce the algorithm with details in Section 3 and 4 and experiments on Bolza surfaces with genus $3$ and Klein quadrics of genus $3$ with several triangulations and graphs in Section 5. In the Appendix, we prove the generalization of Tutte's embedding theorem with an index lemma similar to the work of \cite{GGT}. 

\paragraph{Future Work.}
Morphing geometric graphs on the unit $2$-sphere has been an open problem for a long time. An excellent summary and further research direction on it is given in Section 7 in \cite{EH}. It is also related to the problem to determine the homotopy type of spaces of geodesic triangulations on the unit $2$-sphere \cite{CHHS}. It is also natural to explore whether the current algorithm compute compute linear complexity morphs between geodesic embeddings.

\section{Background}
In this section, we summarize necessary mathematical facts about the Poincaré hyperbolic plane $\D$ and Tutte's embedding theorem on hyperbolic surfaces. 

\subsection{The Poincar\'e disk and M\"obius transformations}

The Poincaré disk model of the hyperbolic plane is the open unit disk $\D=\{z\in\C:|z|<1\},$
equipped with the metric
\[
ds^2=\frac{4|dz|^2}{(1-|z|^2)^2}
=\lambda_z^2|dz|^2,
\qquad
\lambda_z=\frac{2}{1-|z|^2}.
\]
Its geodesics are Euclidean diameters and circular arcs orthogonal to
$\partial\D$. The orientation-preserving hyperbolic isometries are the
M\"obius transformations preserving $\D$, which can be written as
\[
M(z)=\frac{az+b}{\overline{b}z+\overline{a}}, \qquad |a|^2-|b|^2=1.
\]
For $x,y\in\D$, their M\"obius addition is
\[
x \oplus y = \frac{x+y}{1+\overline{x}y},
\]
and their hyperbolic distance is
\[
d_{\D}(x,y) = 2 \operatorname{arctanh}\left|(-x)\oplus y\right|.
\]
The logarithm map $\log_x:\D\to T_x\D$ is
\[
\log_x(y) = \frac{2}{\lambda_x} \frac{\operatorname{arctanh}|\delta|}{|\delta|} \delta, \qquad \delta=(-x)\oplus y,
\]
with $\log_x(x)=0$. The exponential map $\exp_x:T_x\D\to\D$ is
\[
\exp_x(v) = x \oplus \left( \tanh\left(\frac{\lambda_x |v|}{2}\right) \frac{v}{|v|} \right),
\]
with $\exp_x(0)=x$.

Every closed surface of genus $g\geq 2$ admits a hyperbolic metric and can be represented as a quotient
\[
S = \mathbb{D}/\Gamma,
\]
where $\Gamma$ is a discrete group of M\"obius transformations preserving $\mathbb{D}$ acting freely on the Poincar\'e disk. Its fundamental domain is a geodesic polygon in $\mathbb{D}$ whose boundary edges are paired by elements of $\Gamma$ so that the quotient recovers the original surface. Working in a fundamental domain allows one to lift graphs from the surface to the hyperbolic plane, where explicit coordinates and formulas are available. In the rest of the paper, geodesic embeddings and deformations of them on a hyperbolic surface will be represented as a geodesic embedding in $\D$ equivariant with respect to $\Gamma$. 


\subsection{Tutte's embedding theorem for essentially 3-connected graphs}

Tutte's classical embedding theorem states that a 3-connected planar graph can be embedded without edge crossings by fixing the boundary vertices to a convex polygon and placing each interior vertex at the weighted average of its neighbors. This idea extends to triangulations on surfaces of arbitrary topology by replacing Euclidean line segments with geodesic segments and Euclidean convexity with geodesic convexity, including the case of flat tori in \cite{GGT} and the case of hyperbolic surfaces in  \cite{CdV, LWZ}. We generalize Tutte's embedding theorem to a broader family of essentially 3-connected graphs on hyperbolic surfaces, which is suitable for the graph morphing problem.

Assume $S$ is a closed orientable hyperbolic surface. 
A cellular decomposition $\mathcal{C}$ of a surface can be described by the sets of vertices $V$, edges $E$, and faces $F$, where each face is a $2$-cell and a marking homeomorphism $f:|\mathcal{C}|\to S$, where $|\mathcal{C}|$ is the geometric realization of $\mathcal{C}$. 
The one-skeleton of $\mathcal{C}$ is a graph $G = G(V, E)$. This graph might not be simple and may contain multiple edges and loops. A cellular decomposition is essentially 3-connected if the lift $\tilde G$ of $G$ to the universal covering is 3-connected. 

A map $\phi:G\to S$ is an embedding if it is continuous and injective when $G$ is regarded as a topological space and $\phi$ is homotopic to the marking $f$ restricted to $G$. A map $\phi:G\to S$ is a \textit{geodesic mapping} if it maps each edge of $G$ to a geodesic in $S$, where every edge of $G$ is identified with the unit interval $[0, 1]$ and the restriction of $\varphi$ on $[0, 1]$ is a geodesic parametrized with unit speed. 

A geodesic mapping $\phi$ is \textit{$w$-balanced} if there exists a weight function $w$ on directed edges  
such that for any $v_i\in V$ and $x=\phi(v_i)$, the tangent vectors $\vec{v}_{ij}$ of the image of edges $e_{ij}$ incident to $v_i$ satisfies in the tangent space $T_xS$

\begin{equation}
\label{eq:balanced-geodesic-mapping}
\sum_{e_{ij}\in N(v_i)}w_{ij}\vec{v}_{ij} = 0.
\end{equation}

Notice that the weight is, in general, not symmetric. A map $\varphi:G\to S$ is a \textit{geodesic embedding} if it is a geodesic mapping and an embedding. Let $X(S, G, \varphi)$ be the space of geodesic embeddings of $G$ on $S$ in the homotopy class determined by $\varphi$. We refer to \cite{LWZ} for more details about related notions. 

The following theorem provides the mathematical foundations for the algorithms to morph graphs on hyperbolic surfaces. 

\begin{theorem}
\label{tutte}
    Let $G$ be the one skeleton of an essentially 3-connected cellular decomposition $\mathcal{C}$ on a hyperbolic surface $S$. For any positive weight $w$ on directed edges of $G$, there exists a unique balanced geodesic mapping with respect to $w$. Moreover, this geodesic mapping is a geodesic embedding of $G$ on $S$ such that every face of $\mathcal{C}$ is embedded as a strictly convex polygon.
\end{theorem}

One restricted version of the above theorem for triangulations has been proved in \cite{LWZ}, where the existence and uniqueness of the geodesic mapping was established based on the analysis of one-ring neighborhood of the lift in $\mathbb{D}$. Therefore, the argument can be directly generalized to essentially 3-connected cellular decompositions.
In the Appendix, we will prove that balanced geodesic mappings for essentially 3-connected graphs are geodesic embeddings such that every face is embedded as a convex hyperbolic polygon. We will implement a characterization of essentially 3-connected graphs on a surface with non-positive Euler characteristics by \cite{Mohar} and an index theorem for hyperbolic surfaces, generalizing the idea from \cite{GGT}.

\subsection{Hyperbolic mean value coordinates}
Hyperbolic mean value coordinates in \cite{hypermean} generalize classical mean value coordinates in \cite{Floater} from the Euclidean plane to the hyperbolic plane.
Given a convex geodesic polygon and a point in its interior, these coordinates assign positive weights to the vertices that sum to one and reproduce the point as a hyperbolic barycenter. They vary smoothly with the point and can be evaluated using explicit formulas involving hyperbolic distances and angles.
Hyperbolic mean value coordinates provide a flexible tool for interpolation, parameterization, and deformation of geometric data on hyperbolic surfaces. 

Let $P=(v_1,\dots,v_n)$ be a geodesically star-shaped polygon in the hyperbolic plane $\mathbb{H}^2$, and let $x$ be an eye of $P$, namely every geodesic segment from $x$ to a boundary point of $P$ is contained in the interior of $P$.
Denote by $d_i=d_{\mathbb{H}}(x,v_i)$ the hyperbolic distance from $x$ to the vertex $v_i$. Let $\alpha_i$ be the angle at $x$ between the geodesic segments
$[x,v_i]$ and $[x,v_{i+1}]$, where indices are taken modulo $n$. Define the
unnormalized weights
\[
\widetilde{w}_i(x)
=
\frac{
\tan(\alpha_{i-1}/2)+\tan(\alpha_i/2)
}{
\sinh d_i
}.
\]
The hyperbolic mean value coordinates are then given by
\[
w(x, v_i)
=
\frac{\widetilde{w}_i(x)}{\sum_{j=1}^n \widetilde{w}_j(x)}.
\]

These coordinates satisfy the following properties:
\begin{enumerate}
    \item Partition of unity: $w(x, v_i) > 0$ for all $i$, and $\sum_{i=1}^n w(x, v_i)=1.$


    \item Smoothness:
    Each $w(x, v_i)$ is smooth in the interior of the polygon.

    \item Isometry invariance:
    The coordinates are preserved under hyperbolic isometries.
\end{enumerate}

Note that the original definition of hyperbolic mean value coordinates is for points in a polygon $P$. In our algorithm to morph graphs, we compute the mean value coordinates in the star of every vertex $v_k$ with neighbors $v_i$. Then $w(v_k, v_i)$ associated with each ordered pair of $(v_k, v_i)$ can be viewed as a weight function evaluated on the directed edges of $G$.

\section{Methodology}
The basic framework to morph two graphs follows from the work in \cite{FG} and the proof of the contractibility of spaces of geodesic triangulations on surfaces in \cite{EL2021, Luo, LWZ}. \Cref{tutte} defines a continuous map $\Psi$ from the space of positive weights on directed edges to the space of geodesic embeddings $X(S, G, \varphi)$.
Hyperbolic mean value coordinates define a continuous map $\sigma$ from $X(S, G, \varphi)$ to the space of positive weights on directed edges. 
We will use both maps to generate a family of geodesic embeddings connecting two given geodesic embeddings.

\begin{figure}[htbp]
\centering
\begin{tikzpicture}[
    weight/.style={draw, rectangle, fill=red!20, minimum width=1.6cm,
        minimum height=0.8cm, align=center},
        embedding/.style={draw, rectangle, fill=blue!20, rounded corners,
        minimum width=1.8cm, minimum height=1cm, align=center},
    flow/.style={draw, -latex}
]
    \node[weight] (w0) {$w^0_{ij}$};
    \node[weight, right=1.3cm of w0] (w1) {$w^{t_1}_{ij}$};
    \node[right=0.7cm of w1] (wdots) {$\cdots$};
    \node[weight, right=0.7cm of wdots] (wnm) {$w^{t_{N-1}}_{ij}$};
    \node[weight, right=1.3cm of wnm] (wn) {$w^1_{ij}$};

    \node[embedding, below=2cm of w0] (tau0) {$\tau_0$};
    \node[embedding] (tau1) at (w1 |- tau0) {$\Psi(w^{t_1}_{ij})$};
    \node (taudots) at (wdots |- tau0) {$\cdots$};
    \node[embedding] (taunm) at (wnm |- tau0) {$\Psi(w^{t_{N-1}}_{ij})$};
    \node[embedding] (taun) at (wn |- tau0) {$\tau_1$};

    \node[above=0.55cm of wdots] (interpolation)
        {$w^{t_k}_{ij}=(1-t_k)w^0_{ij}+t_kw^1_{ij},\quad t_k=k/N$};

    \path[flow] (w0) -- (w1);
    \path[flow] (w1) -- (wdots);
    \path[flow] (wdots) -- (wnm);
    \path[flow] (wnm) -- (wn);
    \path[flow] (tau0) -- (tau1);
    \path[flow] (tau1) -- (taudots);
    \path[flow] (taudots) -- (taunm);
    \path[flow] (taunm) -- (taun);

    \path[flow] (tau0) -- node[left] {$\sigma$} (w0);
    \path[flow] (taun) -- node[right] {$\sigma$} (wn);
    \path[flow] (w1) -- node[right] {$\Psi$} (tau1);
    \path[flow] (wnm) -- node[left] {$\Psi$} (taunm);
\end{tikzpicture}
\caption{The framework in \cite{FG} evaluated at the discrete time
$0=t_0<t_1<\cdots<t_N=1$.}
\end{figure}
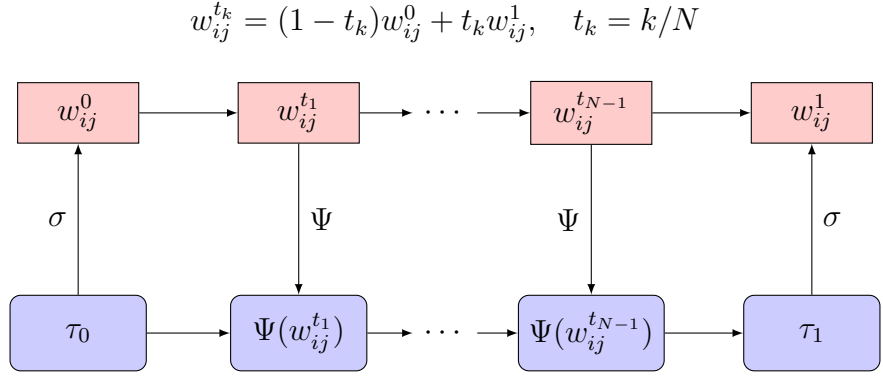
 
Let $\tau_0$ and $\tau_1$ be two geodesic embeddings in $X(S, G, \varphi)$. We find their hyperbolic mean value coordinates through $\sigma$, $w^0_{ij} = \sigma(\tau_0)$ and $w^1_{ij} = \sigma(\tau_1)$. Since the space of weights is convex, a linear interpolation $w^t_{ij} = (1-t)w^0_{ij} + tw^1_{ij}$ produces a family of weights. Then $\Psi(w^t_{ij})$ is a family of geodesic embeddings connecting $\tau_0$ and $\tau_1$. 

\section{Algorithm}
The input of our algorithm consists of two geodesic embeddings of $G$ on $S$, represented in a fundamental polygon of $S$ in $\mathbb{D}$ by two vertex position matrices $Z^0$ and $Z^1$, respectively, of the same size. 

We also store the side-pairing Möbius transformations to encode the side identifications of the cut-open fundamental polygon. Note that they share the same Möbius transformations.
For vertices attached by Möbius transformations, we choose one copy as the representative vertex and call it the \emph{root}; the remaining attached copies are called \emph{slave} vertices. More explicitly, assume a vertex $v_i$ has $m_i$ copies attached by Möbius transformations, and we choose $z_i^1$ as the root copy and $z_i^\ell$ is a slave copy, then the attaching map $M_\ell$ satisfies
$$
z_i^\ell = M_\ell(z_i^1), \qquad \ell=2,\ldots,m_i .
$$
Thus root and slave copies are both stored in $Z^0$ and $Z^1$, but only the root is an independent position variable. An interior root has no slave copies ($m_i=1$). A boundary root may have one or more slave copies attached by Möbius transformations. We will use $v_i$ and $z_i$ interchangeably to denote a vertex in our algorithm when the context is clear.

Our algorithm consists of the following three steps:
\begin{enumerate}
    \item Given the two vertex position matrices $Z^0$ and $Z^1$ and the side-pairing Möbius transformations, compute two directed edge weights $w^0_{ij}$ and $w_{ij}^1$ on them respectively by the hyperbolic mean value coordinates. 
    \item Linearly interpolate the two directed edge weights to obtain a family of directed edge weights $w^t_{ij} = (1-t)w^0_{ij} + tw^1_{ij}$ at each time step $t$.
    \item For each $w^t_{ij}$, compute the $w^t_{ij}$-balanced embedding by a gradient descent algorithm computed on the hyperbolic plane $\mathbb{D}$. 
\end{enumerate}

We now discuss them in details.

\paragraph{Hyperbolic convexity of faces.}
For two inputs of geodesic embeddings of cellular decomposition, the first step of our algorithm requires that higher-order faces, namely faces with more than three edges, of the inputs must be strictly geodesically convex. 
We call them \textit{convex embeddings}. This is checked with tangent directions of Poincaré geodesics each face. 

We first find the unit tangent vector for the geodesic between two adjacent vertices $p=(x_p,y_p)$ and $q=(x_q,y_q)$. If $x_p y_q-y_p x_q=0$, the geodesic is a Euclidean diameter and the tangent at $p$ toward $q$ is
\[
\tau_{p\to q} = \frac{q-p}{\|q-p\|}.
\]
Otherwise the geodesic is the Euclidean circle orthogonal to $\partial\D$ and passing through $p$ and $q$. Its center $c=(c_x,c_y)$ is found by solving
\[
2c\cdot p=\|p\|^2+1,\qquad 2c\cdot q=\|q\|^2+1,
\]
equivalently, for $\Delta=x_p y_q-y_p x_q$,
\[
c_x = \frac{(\|p\|^2+1)y_q-(\|q\|^2+1)y_p}{2\Delta}, \qquad c_y = \frac{x_p(\|q\|^2+1)-x_q(\|p\|^2+1)}{2\Delta}.
\]
The tangent line is perpendicular to the radius $p-c$, so the two possible unit tangents are
\[
\pm\frac{(-(y_p-c_y),\,x_p-c_x)}{\|p-c\|}.
\]
We choose the sign for which $p+\epsilon\tau_{p\to q}$ is closer to $q$ than $p-\epsilon\tau_{p\to q}$ for a small $\epsilon>0$. 

For a face $(v_0,\ldots,v_{\ell-1})$, define the two consecutive unit tangent vectors $a_k=\tau_{v_{k-1}\to v_{k}},\\ b_k=\tau_{v_k\to v_{k+1}}$.
The signed turning angle at $v_k$ is the oriented angle from $a_k$ to $b_k$:
\[
\theta_k = \Arg\!\left( \left\langle a_k,b_k\right\rangle + \sqrt{-1}\,\det(a_k,b_k) \right)\in(-\pi,\pi].
\]
Here, $\left\langle u,v\right\rangle = \operatorname{Re}(u\overline{v})$ and $\det(u, v) = \operatorname{Im}(\bar uv)$. The face is accepted as geodesically convex when all $\theta_k>0$ have the sign of the face orientation.

In practice, a simple method to generate convex embeddings is provided in Section 5 using discrete harmonic maps.

\paragraph{Directed edge weights.}
Given a convex embedding, we compute the weights on directed edges of $G$. Around a root vertex $z_i$, all neighbors are first moved into the same local star using the side-pairing Möbius transformations. Let $z_i$ denote its neighbors. Next, we move the star by a disk isometry sending $z_i$ to the origin. 

For each neighbor occurrence $v_j$, let
\begin{equation}
\label{eq:local-star-coordinate}
u_j=\frac{z_j-z_i}{1-\overline{z_i}z_j} = r_j e^{\sqrt{-1}\varphi_j},
\qquad \varphi_j=\Arg(u_j)\in(-\pi,\pi],
\qquad \rho_j=2\operatorname{arctanh} r_j.
\end{equation}

After grouping neighbors with the same angle, 
write the distinct angles in increasing order as
\[
-\pi<\varphi_0<\varphi_1<\cdots<\varphi_{q-1}\leq\pi.
\]
This gives their cyclic order around $v_i$. 
To include the gap across the branch cut, set
\[
\varphi_{-1}:=\varphi_{q-1}-2\pi,
\qquad
\varphi_q:=\varphi_0+2\pi.
\]
For a neighbor $v_j$ whose direction is the $k$-th distinct angle,
the adjacent angular gaps before and after neighbor $v_j$ are
\[
\alpha_j^{-}:=\varphi_k-\varphi_{k-1},
\qquad
\alpha_j^{+}:=\varphi_{k+1}-\varphi_k,
\qquad
0<\alpha_j^{-}<\pi,
\quad
0<\alpha_j^{+}<\pi,
\]
where the final inequalities follow from strict convexity of the incident faces.
The unnormalized hyperbolic mean value directed weights are
\begin{equation}
\widetilde{w}_{ij} = \frac{ \tan(\alpha_j^{-}/2)+\tan(\alpha_j^{+}/2) }{ \sinh \rho_j }.
\end{equation}
When multiple attached occurrences have the same angular direction, the numerator is shared equally among them. We then normalize the outgoing weights at each center vertex,
\begin{equation}
\label{eq:directed-mean value-weight}
{w}_{ij}=\frac{\widetilde{w}_{ij}}{\sum_{k\sim i}\widetilde{w}_{ik}},
\end{equation}
so the directed coefficients satisfy $\sum_{j\sim i}{w}_{ij}=1$, where $j\sim i$ means $j$ is adjacent to $i$.
\begin{algorithm}[htbp]
\caption{Directed mean value weights on the attached local star}
\label{alg:directed-mean-value-weights}
\begin{algorithmic}[1]
\State \textbf{Input:} Graph $G=(V,E)$, position array $Z$, side-pairing maps
\State \textbf{Output:} Directed weights ${w}_{ij}$ for all oriented edge incidences
\ForAll{root vertices $v_i\in V$}
    \State Build the attached local star of $v_i$ using $Z$ and the side-pairing maps
    \ForAll{attached neighbor occurrences $v_j$}
        \State Compute $u_j,\phi_j,\rho_j$ by Eq.~\eqref{eq:local-star-coordinate}
    \EndFor
    \State Sort by $\phi_j$ and compute the adjacent angles $\alpha_j^{-},\alpha_j^{+}$
    \State Compute the normalized weights: ${w}_{ij}$ by Eq.~\eqref{eq:directed-mean value-weight}
\EndFor
\State \Return all directed weights ${w}_{ij}$ and ${w}_{ji}$
\end{algorithmic}
\end{algorithm}

\paragraph{Iterative update of vertices via gradients.}
Given a positive weight $w$ on directed edges and an initial embedding represented by $Z$, we find a $w$-balanced geodesic embedding by a local gradient descent algorithm as follows.

\textbf{Case 1}: For a root vertex $v_i$ without slave copies, consider the function 
\begin{equation}
\label{eq:local-convex-function}
F_i(y)=\sum_{j\sim i}{w}_{ij}\bigl(\cosh d_{\D}(y,z_j)-1\bigr).
\end{equation}

Its gradient at $z_i$ is
\begin{equation}
\label{eq:directed-edge-gradient}
\nabla F_{i} (z_i)
=-\sum_{j\sim i}{w}_{ij}
\frac{\sinh d_{ij}}{d_{ij}}\log_{z_i}(z_j)
:=g_i \in T_{z_i}\D,
\qquad d_{ij}=d_{\D}(z_i,z_j),
\end{equation}
The root is then updated iteratively by $z_i\gets\exp_{z_i}(-\eta g_i)$ for a step size $\eta$.

Eq.~\eqref{eq:directed-edge-gradient} provides a local gradient for each vertex and we use gradient descent to find the $w$-balanced embedding. Indeed, $\log_{z_i}(z_j)=d_{ij}\vec e_{ij}$ with $\vec e_{ij}$ is the unit tangent vector at $z_i$ pointing to $z_j$ , so $\nabla F_i(z_i)=0$ is equivalent to the $w$-balance condition in Eq.~\eqref{eq:balanced-geodesic-mapping}.
Notice that $F_i$ is defined locally for each vertex $i$ and its star, and in general $w_{ij} \neq w_{ji}$. 

\textbf{Case 2}: For a root vertex $v_i$ with $m$ slave copies, the goal is to update the root copy $z_i^1$ and its slave copies $z_i^2,\ldots,z_i^m$ respecting the constraints $z_i^\ell=M_\ell(z_i^1)$.

Assume we have computed gradient vectors $g_i^\ell\in T_{z_i^\ell}\D$ by Eq.~\eqref{eq:directed-edge-gradient} on all slave vertices $z_i^l$. We need to pull back thees gradients to the root vertex. The accumulated gradient $G_i$ on the root vertex can be found by the sum of pull-back gradients
\[
G_i=g_i^1+\sum_{\ell=2}^{m}\widetilde g_i^\ell,
\qquad
\widetilde g_i^\ell=d(M_\ell^{-1})_{z_i^\ell}\bigl(g_i^\ell\bigr),
\]
where the differential $d(M_\ell^{-1})_{z_i^\ell}$ is computed below. We update the root by
\[
z_i^1\gets\exp_{z_i^1}(-\eta G_i),
\]
and then reset each slave by $z_i^\ell\gets M_\ell(z_i^1)$.

For a root $z_r$ and one of its slave 
$$
z_s=M(z_r) = \frac{az_r+b}{\overline{b}z_r+\overline{a}},
$$
let $g_s\in T_{z_s}\D$ be the gradient contribution at the slave copy. The corresponding contribution at the root is the metric adjoint
\[
g_{r\leftarrow s} = (dM_{z_r})^{*}g_s.
\]
Because $M$ is a Möbius transformation, this simplifies to the inverse differential,
\[
(dM_{z_r})^{*}g_s = d(M^{-1})_{z_s}(g_s) = (M^{-1})'(z_s)\,g_s.
\]
Computationally, for $M(z)=(az+b)/(\overline{b}z+\overline{a})$, the inverse is
\[
M^{-1}(z)=\frac{\overline{a}z-b}{-\overline{b}z+a},
\]
and its derivative is
\[
(M^{-1})'(z)=\frac{1}{(-\overline{b}z+a)^2}.
\]
Thus the pulled-back gradient contribution used in the implementation is
\begin{equation}
\label{eq:pullback-gradient}
g_{r\leftarrow s} = d(M^{-1})(z_s)g_s = \frac{g_s}{(-\overline{b}z_s+a)^2}.
\end{equation}

\begin{algorithm}[htbp]
\caption{Gradient accumulation and update at one root}
\label{alg:root-slave-gradient-accumulation}
\begin{algorithmic}[1]
\State \textbf{Input:} Root copy $z_i^1$, slave copies $z_i^2,\ldots,z_i^m$, attaching maps $M_\ell$ satisfying $z_i^\ell=M_\ell(z_i^1)$, directed weights $\omega$, and step size $\eta$
\State \textbf{Output:} Updated root and slave copies for vertex $i$
\For{$\ell=2,\ldots,m$}
    \State Enforce the slave constraint: $z_i^\ell\gets M_\ell(z_i^1)$
\EndFor
\State $G_i\gets 0$
\For{$\ell=1,\ldots,m$}
    \State Compute $g_i^\ell$ at copy $z_i^\ell$ from Eq.~\eqref{eq:directed-edge-gradient}
    \If{$\ell=1$}
        \State $G_i\gets G_i+g_i^\ell$ \Comment{root-copy contribution}
    \Else
        \State Pull back the slave contribution by Eq.~\eqref{eq:pullback-gradient}: $\widetilde g_i^\ell\gets d(M_\ell^{-1})_{z_i^\ell}(g_i^\ell)$
        \State $G_i\gets G_i+\widetilde g_i^\ell$
    \EndIf
\EndFor
\State Update the root: $z_i^1\gets \exp_{z_i^1}(-\eta G_i)$
\For{$\ell=2,\ldots,m$}
    \State Enforce the slave constraint: $z_i^\ell\gets M_\ell(z_i^1)$
\EndFor
\State \Return updated copies $z_i^1,\ldots,z_i^m$
\end{algorithmic}
\end{algorithm}

\paragraph{Morphing algorithm.}
Our algorithm produces $N$ frames from the initial positions $Z^0$ to the final positions $Z^1$. For $t_k=k/(N-1)$, $k=0,\ldots,N-1$, we linearly interpolate the directed weights,
\[
w_{ij}^{t_k} = (1-t_k)\,w_{ij}^{0} + t_k\,w_{ij}^{1},
\]
where $w_{ij}^{0}, w_{ij}^{1}$ are the directed edge weights computed on $Z^0$ and $Z^1$ by \Cref{alg:directed-mean-value-weights}.
Each intermediate frame $Z^{t_k}$, is obtained by minimizing the local functions in \Cref{eq:local-convex-function} for $w^{t_k}$, initialized from $Z^{t_{k-1}}$. Algorithm~\ref{alg:root-slave-gradient-accumulation} updates every root and resets its slave copies using the attaching maps, thereby preserving the attaching constraints. In our experiments, the iteration stops when the maximum vertex change in one update is below $\varepsilon = 10^{-8}$, or when the iteration budget $K_{\max} = 10k$ is reached.

\begin{algorithm}[H]
\caption{Mean value coordinate morphing}
\label{alg:mean value-morph}
\small
\begin{algorithmic}[1]
\State \textbf{Input:} Graph $G=(V,E)$, endpoints $Z^0,Z^1$, frame count $N$, step size $\eta$, iteration budget $K_{\max}$, and tolerance $\varepsilon$
\State \textbf{Output:} Position frames $Z^{t_k}, k = 0, \ldots, N-1$
\State Compute directed weights $\omega^0$ from $Z^0$ by Algorithm~\ref{alg:directed-mean-value-weights}
\State Compute directed weights $\omega^1$ from $Z^1$ by Algorithm~\ref{alg:directed-mean-value-weights}
\State $Z^{t_0}\gets Z^0$ and $Z^{t_{N-1}}\gets Z^1$
\For{$k=1,\ldots,N-2$}
    \State $t_k\gets k/(N-1)$
    \State $\omega(t_k)\gets(1-t_k)\omega^0+t_k\omega^1$
    \State Initialize the current frame with $Z\gets Z^{t_{k-1}}$
    \For{$s=1,\ldots,K_{\max}$}
        \State $Z_{\mathrm{old}}\gets Z$
        \ForAll{root vertices $i$}
            \State Apply Algorithm~\ref{alg:root-slave-gradient-accumulation} to the root copy and all slave copies of $i$ using $\omega(t_k)$
            \State Store the returned root and slave positions in $Z$
        \EndFor
        \State $\delta\gets\max_i d_{\D}(Z_i,Z_{\mathrm{old},i})$
        \If{$\delta<\varepsilon$} \State \textbf{break} \EndIf
    \EndFor
    \State Store $Z^{t_k}\gets Z$
\EndFor
\State \Return $Z^{t_k}, k = 0, \ldots, N-1$
\end{algorithmic}
\end{algorithm}

\section{Experiment}
Our implementation is publicly available online.\footnote{\url{https://github.com/YuanL12/hyperbolic-shape-morphing}}
All experiments were run on a MacBook Pro with an Apple M4 Pro chip and $48$\,GB of memory.
We test triangulated and cellular graphs with quadrilateral and pentagonal faces on two cut-open surfaces: the genus--2 Bolza surface, represented by a regular octagonal fundamental domain, and the genus--3 Klein quartic, represented by a regular $14$-gon. The combinatorial sizes of the four test cases are listed in Table~\ref{tab:morph-topology-counts}. We next describe the generation of their synthetic endpoint embeddings; in applications, these endpoints may instead come from prescribed data.

\begin{table}[H]
\centering
\caption{Topology of the four morphing experiments.}
\label{tab:morph-topology-counts}
\begin{tabular}{lrrrrrr}
\hline
Example & Genus & Vertices & Edges & Triangles & Quads & Pentagons \\
\hline
Triangulation & 2 & 190 & 527 & 338 & 0 & 0 \\
Cells & 2 & 190 & 490 & 270 & 25 & 6 \\
Triangulation & 3 & 250 & 677 & 428 & 0 & 0 \\
Cells & 3 & 250 & 557 & 208 & 80 & 20 \\
\hline
\end{tabular}
\end{table}

\paragraph{Generating two convex embeddings $Z^0$ and $Z^1$.}
We construct convex embeddings by discrete harmonic maps as minimizers of the discrete Dirichlet energy \cite{CdV, Gaster, Lam, ZLL}. The discrete Dirichlet energy is determined by a symmetric (or undirected) weight $w_{ij}$,
$$
E=\frac12\sum_{\{i,j\}\in E}w_{ij}\,l_{ij}^2,
$$
where $l_{ij}$ is the length of the edge. 
In the universal cover $\D$, we write a vertex position as a complex number $z_i \in \D$. For undirected edge weights $w_{ij} = w_{ji}$, the hyperbolic Dirichlet energy is 
\begin{equation}
\label{scalar-hyperbolic-dirichlet-energy}
E=\frac12\sum_{\{i,j\}\in E} w_{ij}\,d_{\D}(z_i,z_j)^2.
\end{equation}
It is minimized by the same gradient descent framework.
It has been shown in \cite{CdV} that there exists a unique minimizer producing a convex embedding of the graph on surface as discrete harmonic maps  \cite{CdV, Lam, ZLL}.

For each of four graphs, we compute two convex embeddings by minimizing the scalar hyperbolic Dirichlet energy in Eq.~\eqref{scalar-hyperbolic-dirichlet-energy}. The initial convex embedding uses undirected edge weights $w_e \equiv 1$. 
The final convex embedding uses a non-uniform weight in which the first half of the edges have weight $10$ and the rest have weight $1$. The resulting minimizers $Z^0$ and $Z^1$ are the morph endpoints. Afterward, the undirected weights are discarded and the morphing algorithm computes directed mean value weights from the two convex embeddings.


All four morphs were generated by $N = 25$ frames, with selected frames shown in \Cref{fig:g2-triangulation-morph,fig:g2-cells-morph,fig:g3-triangulation-morph,fig:g3-cells-morph}. Green curves are graph edges, the solid black curve is the fixed fundamental domain as a reference. We use label the boundary vertices by red to highlight the change of vertex positions. We also use orange and green faces to indicate quadrilateral and pentagonal faces respectively.

Across all four examples, every generated frame remained embedded in the Poincaré disk. No frame contained crossing of geodesic edges. For the two cellular morphs, every displayed and generated frame also passed the convexity check on all quadrilateral and pentagonal faces. These experiments indicate that the algorithm to morph graphs applies to triangulations and cellular decompositions with convex higher-order faces, on both the genus--2 Bolza surface and the genus--3 Klein quartic.

\textbf{Acknowledgment.} This work is a collaboration during the workshop ``Foundations of Computational Geometry and Topology'' at ICERM, Brown University in May 18-21, 2026. 

\begin{figure}[H]
\centering
\includegraphics[width=\linewidth]{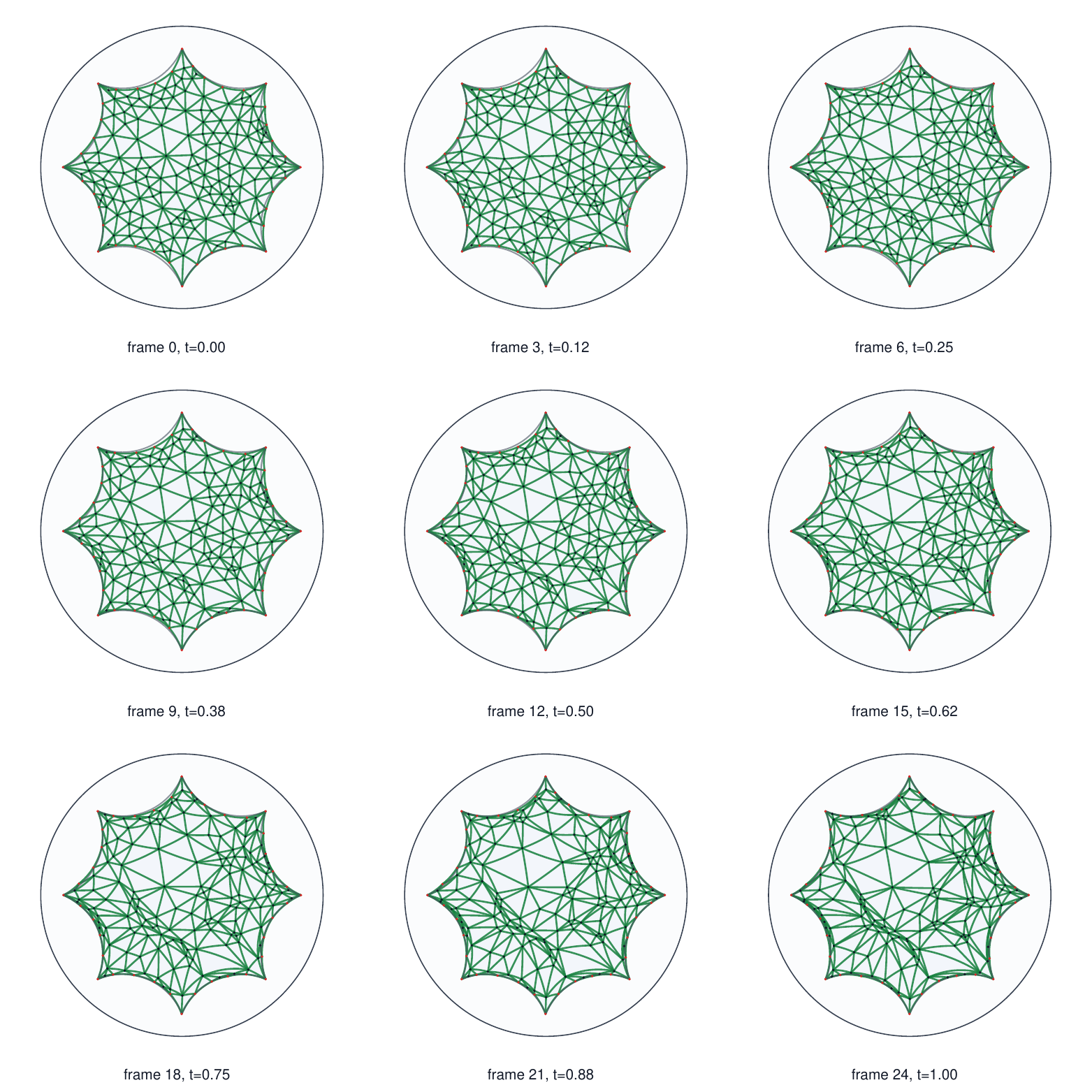}
\caption{Genus--2 triangulation morph. It has 190 vertices, 527 edges, and 338 triangular faces. The frames show the normalized mean value coordinate morphing from the uniform initial embedding to the non-uniform final embedding.}
\label{fig:g2-triangulation-morph}
\end{figure}

\begin{figure}[H]
\centering
\includegraphics[width=\linewidth]{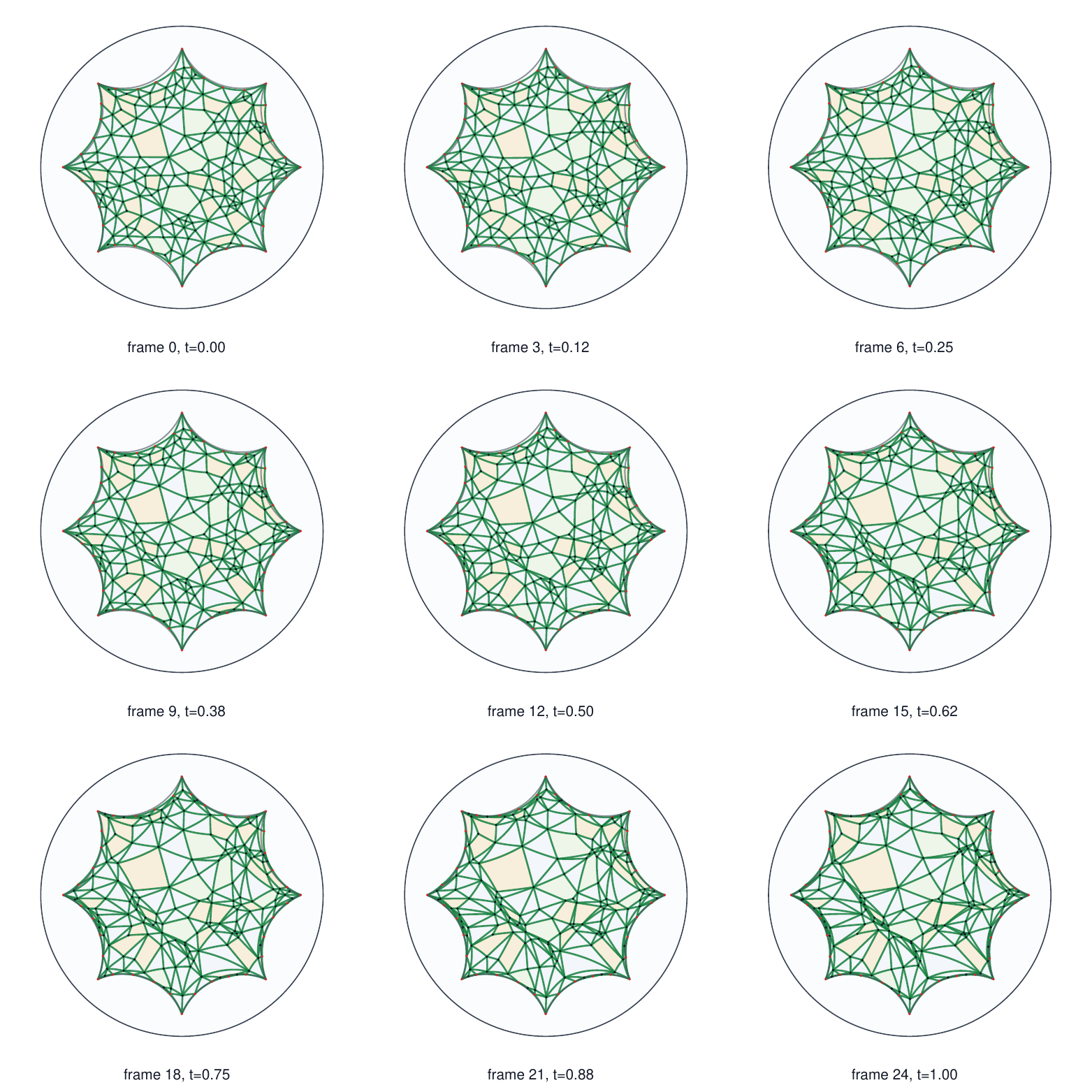}
\caption{Genus--2 cellular morph. It has 190 vertices, 490 edges, 270 triangles, 25 quadrilaterals, and 6 pentagons. The displayed normalized mean value frames preserve geodesic convexity for all quadrilateral and pentagonal cells.}
\label{fig:g2-cells-morph}
\end{figure}

\begin{figure}[H]
\centering
\includegraphics[width=\linewidth]{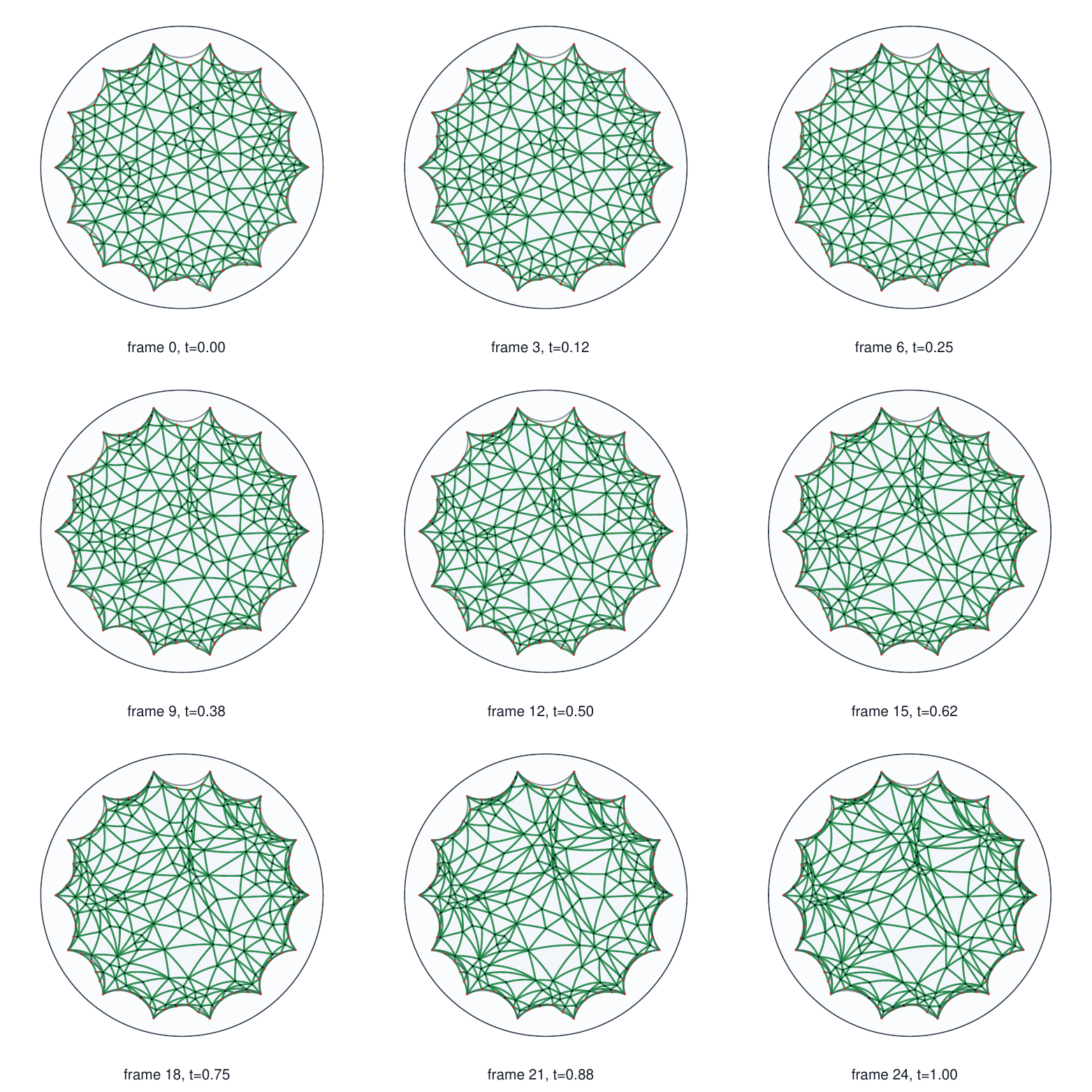}
\caption{Genus--3 triangulation morph on the Klein quartic fundamental domain. It has 250 vertices, 677 edges, and 428 triangular faces. The frames show the normalized mean value coordinate morph from the uniform initial embedding to the non-uniform final embedding.}
\label{fig:g3-triangulation-morph}
\end{figure}

\begin{figure}[H]
\centering
\includegraphics[width=\linewidth]{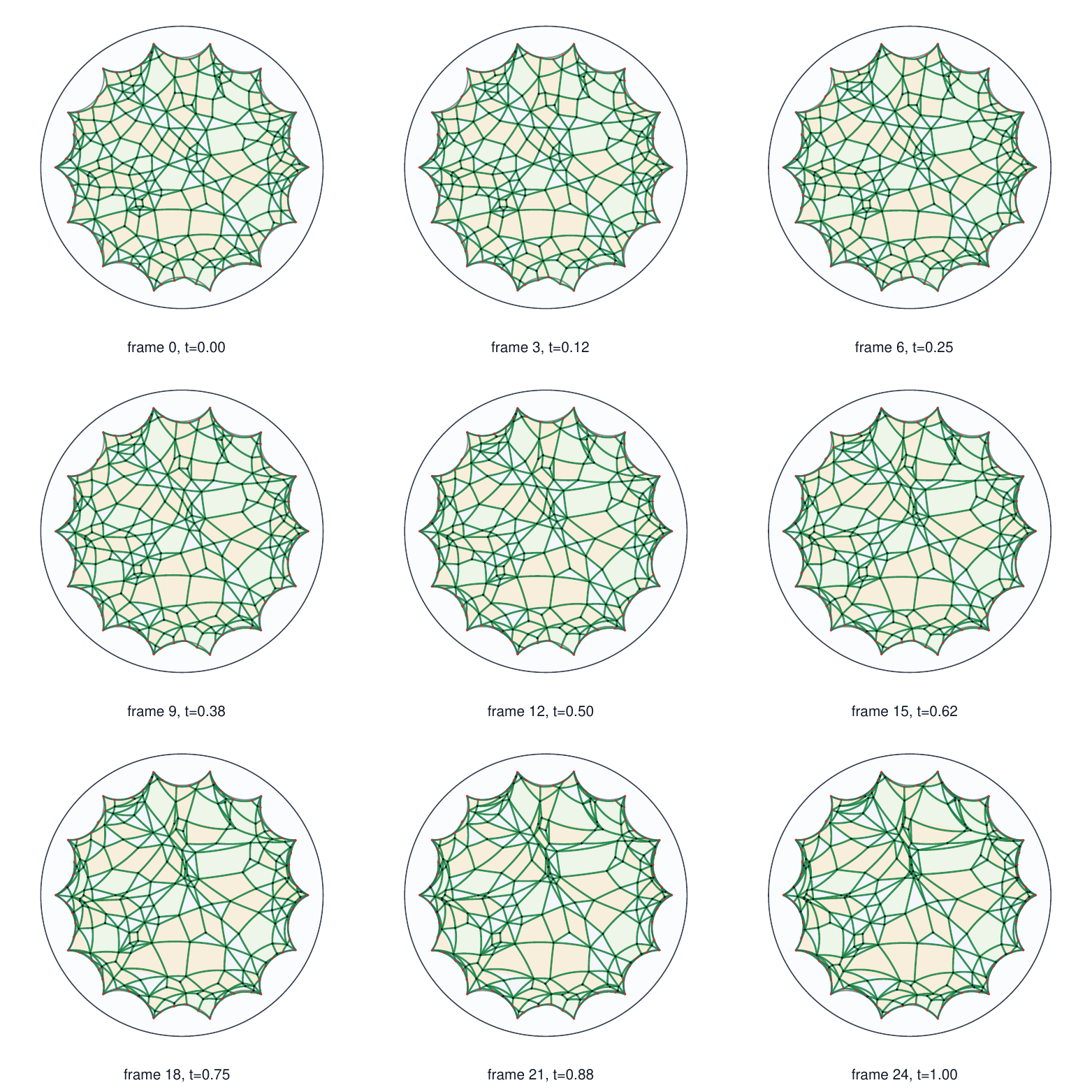}
\caption{Genus--3 cellular morph. It has 250 vertices, 557 edges, 208 triangles, 80 quadrilaterals, and 20 pentagons. As in the genus--2 cellular example, normalized mean value weights are computed from the local stars, and the displayed frames preserve geodesic convexity for all higher-order cells.}
\label{fig:g3-cells-morph}
\end{figure}

\newpage
\section{Appendix: Theorem \ref{tutte} for essentially 3-connected graphs}

The goal of this appendix is to prove a slight generalization of Theorem 1.6 in \cite{LWZ}.
\begin{theorem}
\label{main}
    For every essentially $3$-connected cellular decomposition on a closed hyperbolic surface $S$, a balanced geodesic mapping from the one skeleton of the cellular decomposition to $S$ is an embedding such that each $2$-cell is embedded as a strictly convex hyperbolic polygon. 
\end{theorem}

The key idea of the proof is an index theorem for discrete one forms in \cite{GGT}, which works for locally flat surfaces. Colin de Verdière \cite{CdV} introduced the degenerate angle structures on triangulations of a geodesic mapping to hyperbolic surfaces and proved a similar index theorem. We generalize this concept of the angle structures to cellular decompositions. Here, we identify every edge with the interval $[0, 1]$ and $\varphi_{ij}(t)$ for $t\in[0, 1]$ is a geodesic segment on $S$, where $\varphi_{ij}:G\to S$ is a geodesic mapping  . Let $l_{ij}$ be the length of the image of $\varphi_{ij}$.

Fix such a $\varphi$. There is a natural extension $\tilde \varphi:|\mathcal{C}|\to S$ of $\varphi$ from $G$ to $\mathcal{C}$ such that the restriction of $\tilde\varphi$ on each cell $\sigma$ lifts to $\Phi:\sigma\to \tilde S$, which is the development of the geodesic polygon $\tilde\varphi(\sigma)$ to the universal covering $\tilde S$. In the universal cover, $\Phi_\sigma(\sigma)$ is a potentially degenerate or immersed geodesic polygon. 

Recall that a corner of $\mathcal{C}$ is a pair $(v, \sigma)$, where $v$ is a vertex and $\sigma$ is a cell. A corner can also be indexed by two incident edges $ij$ and $ik$ with $i = v$ and $\sigma$ is a unique cell containing $i, j, k$. An angle structure of $\varphi$ is a map from the set of corners of a cellular decomposition $\mathcal{C}$ to $[0, \pi]$, such that 
\begin{itemize}
    \item if $l_{ij}$ and $l_{ik}$ form a corner and they are not zero, then $\theta_{jk}^i$ is a geometric angle determined by the tangent vectors of these two geodesic at $\varphi(i)$, (we call this \textit{nondegenerate corner}.)
    \item if $l_{ij} = 0$ but $l_{ik}\neq 0$, then the angle $\theta_{jk}^i = \pi/2$. 
    \item if $l_{ij} = l_{ik} = 0$ and the cell $\sigma$ containing $ijk$ is degenerate to a line, which means $\varphi$ sends all the edges of $\sigma$ to a geodesic, then the angles in this cell is arbitrarily assigned as long as the sum is $(n-2)\pi$ if $\sigma$ is an $n$-gon. 
    \item if $l_{ij} = l_{ik} = 0$ but the cell $\sigma$ containing $ijk$ is not degenerate to a line,  then $\theta_{jk}^i = \pi$. 
    
\end{itemize}

From these definitions, we have 
\begin{lemma}
\label{oneface}
    Given a geodesic mapping $\varphi:G\to (S, h)$ and an angle structure $\theta$ on $\mathcal{C}$, 
    $$(n-2)\pi \geq \sum_{i\in\sigma}\theta_\sigma^i - \int_{\Phi_\sigma(\sigma)}K d\tilde A\geq \sum_{i\in\sigma}\theta_\sigma^i - \int_{\tilde\varphi(\sigma)}K dA, $$
    where $dA$ and $d\tilde A$ are are form on $(M, g)$ and its universal covering respectively. 
\end{lemma}
\begin{proof}
    The first inequality is a result of the the Gauss-Bonnet theorem. If $\Phi(\sigma)$ is a homeomorphism so that it develops to a geodesic polygon embedded in $\tilde S$, then this is exactly Gauss-Bonnet formula 
    $$(n-2)\pi - \sum_{i\in\sigma}\theta_\sigma^i = -\int Kd\tilde A$$
    It is straightforward to check that the same is true it some edges are sent to points, where $\theta_\sigma$ is determined above. The same is true by the definition of angle structures if $\sigma$ degenerates to a point.

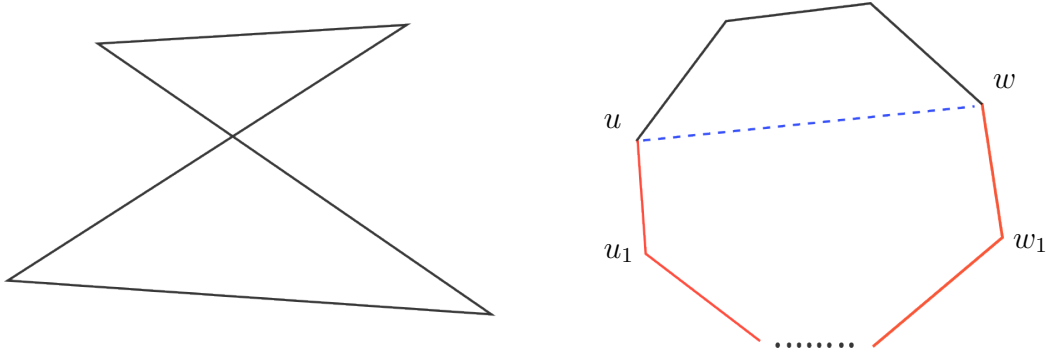
\begin{figure}[H]
\centering		
\begin{overpic}[width=0.8\textwidth]{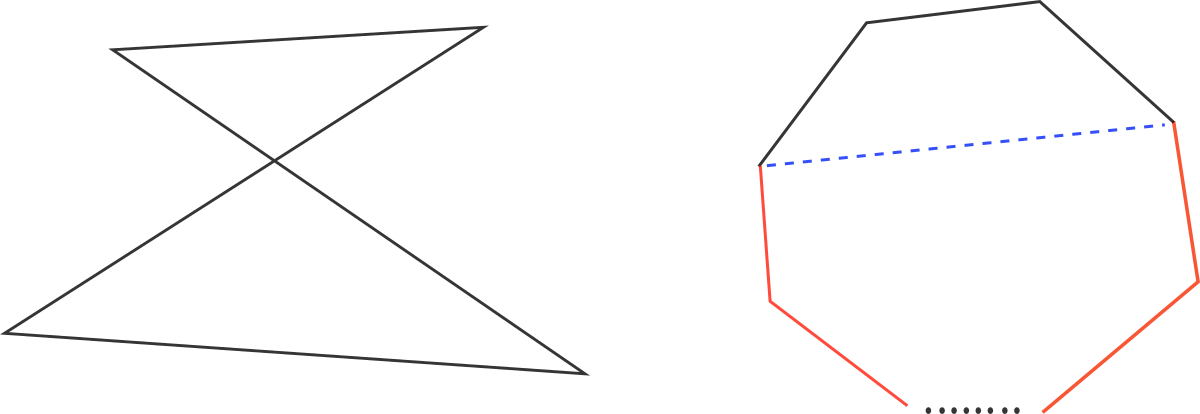}
    \put(60,22){$u$}
    \put(99,26){$w$}
    \put(60,9){$u_1$}
    \put(101,10){$w_1$}
\end{overpic}
\caption{Left: Polygon with intersections. Right: A face of 2nd type with red edges in $H$. }
\label{fig2}
\end{figure}

    However, the equality will not hold if the polygon is immersed. See Figure \ref{fig2}. In this case, the region for the integral is a union of geodesic polygons enclosed by the geodesic segments from the Jordan curve theorem. However, the angles in angles in the sum  $\sum_{i\in\sigma}\theta_\sigma^i $ does not include the angles created by the self-intersections of $\Phi(\sigma)$. Hence, we have 
    $$(n-2)\pi - \sum_{i\in\sigma}\theta_\sigma^i > -\int Kd\tilde A.$$

    The second inequality is immediate since when we project from $\Phi(\sigma)$ to $\tilde \varphi(\sigma)$, the image could overlap so the region for the integration could be smaller. Since $K\leq 0$, we have 
    $$-\int_{\Phi_\sigma(\sigma)}K d\tilde A\geq - \int_{\tilde\varphi(\sigma)}K dA,$$
    so the inequality follows. 
 \end{proof}
Notice that for triangles, it is either degenerate to a line or embedded, so the strict inequality in the first part never shows up. We prove the same results as Lemma 3.3 of \cite{LWZ}, which can be interpreted as an index theorem.
\begin{lemma}
\label{index}
    Let $\theta$ be an angle structure on $\mathcal{C}$. For any $i\in V$, we have 
    $$\sum_{i\in\sigma}\theta_\sigma^i = 2\pi. $$
    And $\Phi(\sigma)\cap\Phi(\sigma')$ has zero area for any $\sigma, \sigma'$ in the cellular decomposition, where $\Phi$ is a lift of $\varphi$ to the universal cover restricted to the corresponding faces.  
\end{lemma}
\begin{proof}
    From the balanced condition, it is clear that all the tangents vectors to geodesics from any $i$ can not be contained in any open half-plane. Then we have from the definition of angle structure that 
    $$\sum_{i\in\sigma}\theta_\sigma^i\geq 2\pi$$
    even if there are degenerate edges incident to $i$. Then 
    $$\sum_{i\in V}(2\pi-\sum_{i\in\sigma}\theta_\sigma^i)+\sum_{\sigma\in F}\int_{\tilde\varphi(\sigma)}KdA \leq \sum_{\sigma\in F}\int_{\tilde\varphi(\sigma)}KdA\leq \int_{S}KdA = 2\pi\chi(S)$$
    The second inequality holds since $\varphi$ is surjective, and it is a strict inequality if and only if there is a face overleap with a face over positive area.  

    On the other hand, we have 
    $$\sum_{i\in V}(2\pi-\sum_{i\in\sigma}\theta_\sigma^i)+\sum_{\sigma\in F}\int_{\tilde\varphi(\sigma)}KdA =  2\pi|V|-\sum_{\sigma\in F}(\sum_{i\in\sigma}\theta_\sigma^i - \int_{\tilde\varphi(\sigma)}KdA )$$
    By Lemma \ref{oneface},  
    $$\sum_{i\in V}(2\pi-\sum_{i\in\sigma}\theta_\sigma^i)+\sum_{\sigma\in F}\int_{\tilde\varphi(\sigma)}KdA \geq 2\pi|V| - \sum_{\sigma\in F}(n_\sigma - 2)\pi = 2\pi\chi(S)$$
    where we use $\chi(S) = V-F+E$. Hence, all the inequalities above are equalities, so we have the conclusion. 
\end{proof}

The same is true for $\Phi$ if we lift to the universal covering. Then we have an infinite planar graph in $\tilde S$ such that the intersection of any two cells has zero area, and the angle sum at each vertex equals $2\pi$. Moreover, if two edges of a cell intersects at their interior, then they are contained in a geodesic. Namely, immersed polygons Figure \ref{fig2} will not show up since the inequality is strict in this immersed polygon.

\subsection{Local Degeneration}
We will first discuss the scenario where one edge is mapped to a point by $\varphi$. Let $i$ be a degenerate vertex if some edge $l_{ij}$ has length zero. Two degenerate vertices $v_1, v_2$ are equivalent if there exists a path in $G$ starting from $v_2$ and ending at $v_2$ such that all the vertices on this path are mapped to the same point $y\in M$ by $\varphi$, which means that all the vertices in this path are degenerate. Let $G'$ be the subgraph generated by such an equivalence class of vertices, and notice that $G'$ is a connected graph. Let $\mathcal{C}'$ be the cell complex by gluing the two cell to $G'$ as how it is glued in $\mathcal C$ if all the vertices of this cell are in $G'$. 

\begin{lemma}
    The cell complex $\mathcal{C}'$ is contractible. 
\end{lemma}
\begin{proof}
    This is part (c) of Lemma 3.5 in \cite{LWZ} or Lemma in page 206 of \cite{CdV}. It is a simple fact since $\tilde\varphi$ restricted to $\mathcal{C}'$ is a homotopy equivalent and $\tilde\varphi(\mathcal{C}')$ is a point. 
\end{proof}

\begin{lemma}
    The graph $G'$ is a path. 
\end{lemma}
\begin{proof}
    A vertex $i$ of $G'$ is called extreme if the cyclic indices $i_1, i_2, i_3, \cdots, i_k, i_{k+1} = i_1$ of its neighbors is a union of two subsets of consecutive indices such that one subset contains edges in $G'$ and another subset contains edges not in $G'$. Intuitively, this means that the link of $v$ intersects with $\mathcal{C}'$ in a connected set. 

    We claim that there are at most two extreme vertices in $G'$. Indeed, by the balanced condition, the sum angles at non-degenerate corners centered at an extreme vertex is at least $\pi$, otherwise all the tangent vectors of non-degenerate edges from $v$ are contained in an open half-plane. 

    Consider all the sum of angles centered at vertices in $G'$ on non-degenerate corners. (Notice that this is not the sum of all angles at the boundary vertices of $G'$.) Since there can not be any overlapping of cells under $\tilde\varphi$. The sum of these angles is at most $2\pi$. Namely, the situation of overlaps of angles in Figure $1$ in \cite{CdV} will not happen.

    Hence, if there are more than two extreme vertices, then the angle sum will be at least $3\pi$, which is a contradiction. 

    Notice that $G$ and $G'$ can not have degree one vertex by the 3-connectivity. Then it is straightforward to see that the $G'$ has to be a path since it is contractible. 
\end{proof}
This also implies that all the angle at non-degenerate corners centered at vertices other than the extreme vertices must be zero. This is also the result in Lemma 3.5 of \cite{LWZ}. For example, the angle $\angle i_1j_2i_6$ in Figure $2$ in \cite{CdV} must be zero. 
\begin{corollary}
\label{linedeg}
    The edges adjacent to vertices in $G'$ are mapped to a geodesic. 
\end{corollary}

Notice that for non-degenerate corners, the image of the two edges of the corner can still be in one geodesic. In this case, we define that a corner of $\varphi$ is singular if it is not degenerate(two edges are not zero) but the angle at this corner is $0$ or $\pi$. Then Lemma 3.6 in \cite{LWZ} still holds in this case. 
\begin{lemma}
\label{singular}
    If $(i, \sigma)$ is a singular corner and $\theta_\sigma^i = \pi$, then all the edges incident to $i$ is mapped to a geodesic by $\varphi$. 
\end{lemma}
\begin{proof}
    If not, then all the tangent vectors at $i$ are contained in a closed half-space but they span the tangent space. This contradicts the balanced condition. 
\end{proof}

\subsection{Proof of Theorem \ref{main}}
The proof of Theorem \ref{main} is based on one characterization of essentially 3-connected graph in Proposition 3.1 of \cite{Mohar}. We rephrase here as following:
\begin{theorem}
\label{3conn}
    Let $G$ be the graph of a cellular decomposition of $S$ such that $\chi(S)\leq 0$. Then $G$ is essentially $3$-connected if and only if the following properties hold:
    \begin{enumerate}
        \item $G$ has no vertices of degree less than $3$,
        \item $G$ has no faces of size less than $3$,
        \item $G$ does not contain vertices $x, y$ and two internally disjoint paths $P_1, P_2$ from $x$ to $y$ such that the closed walk $P_1P_2^{-1}$ bounds a disk $D$ on $S$ and the only vertices on $P_1\cup P_2$ that have a neighbor out of $D$ are $x$ and $y$.
    \end{enumerate}
\end{theorem}

\begin{proof}
    We claim that if $\varphi$ is not an embedding, then there exists a singular corner in the angles structure defined in the previous section. Indeed, if one edge is mapped to a point by $\varphi$, then by Corollay there is a singular corner at the vertex not in $G'$ but adjacent to $G'$. If all the edges have positive length under $\varphi$, and there is no singular corner, then all the $\tilde\varphi$ from $\mathcal{C}$ to $S$ is a local homeomorphism around every vertex, face, and edge by Lemma. Hence, it is a homeomorphism since it is a degree-one map. Then by the balance condition, the geodesic polygons are convex. 

    Define a subgraph $H$ of $G$ as the following: a vertex $v\in V(H)$ if there is a face $\sigma$ of $\mathcal{C}$ such that the corner $(v, \sigma)$ is singular or degenerate, or there is an edge $vu$ from $v$ to a vertex $u\in V(H)$ such that there exists a singular or degenerate corner $vuw$ at $u$. Let $H$ be a subgraph generated by these vertices. Without loss of generality, we assume that $H$ is a connected subgraph of $G$. Notice that $\varphi$ maps $H$ to a geodesic $\lambda$. Moreover, if there is an edge with length zero, then the graph $G'$ defined in the previous section is a subgraph of $H$.

    Given $H$, there are three types of faces of $\mathcal{C}$. A face of the first type intersects with $H$ at no edges or just one edge. A face of the second type intersects with $H$ at more than one edge, but contains some edge not in $H$. A face of the third type has all the edges contained in $H$. Then we can glue all the faces of the third type to $H$ to form a complex $\mathcal{H}$. 

    Given a face of the second type, for convenience, we locally lift it to the universal cover as a polygon $\sigma$. Notice that by Lemma \ref{index}, edges of $\sigma$ can not intersect unless they intersects at some vertex. It can contain some degenerate edges (length zero), degenerate corners, or singular corners. 
    
     Then let $v\in \sigma\cap H$, let $u, w$ be two vertices of $\sigma$ such that the path $u, u_1, \cdots, v, \cdots,w_1, w$ is the longest path consisting of a consecutive vertices of $\sigma\cap H$. Namely, the vertices next to $u$ or $w$ are not in $H$. See Figure \ref{fig2}.

     This means that $u$ and $w$ are boundary vertices of $\mathcal{H}$. Notice that by the definition of the face of the second type, the length of the path $u, u_1, \cdots, v, \cdots,w_1, w$ is at least two, (which means it contains at least two edges.) Two important properties of $u$ and $w$ are 
     \begin{itemize}
         \item The tangent vectors from $u$ span the tangent plane, otherwise all the neighbors of $u$ are in $\lambda$, contradicting to the assumption. The same statement holds for $w$. 
         \item All the edges incident to $u$ and $w$ have positive length, otherwise it is contradicting to Lemma \ref{linedeg}. 
     \end{itemize}

     We will subdivide $\sigma$ by connecting $u$ and $w$ in $\sigma$. Then we extend $\varphi:G\to (S, h)$ to $\phi: \tilde G\to (S, h)$ where $\tilde G$ is the union of $G$ with the edge $uw$ such that $\phi(uw)$ lifts to the geodesic segment in the lift $\sigma$ in the universal covering.We also extend the cell complex from $\mathcal{C}$ to $\tilde{\mathcal{C}}$. Notice that $\phi$ is always a balanced geodesic map. Indeed, it only differs from $\varphi$ on $uw$. Then the tangent vector of the edge $uw$ can be expressed as a linear combination of other vectors, so it is easy to see that we can find weights on $uw$ and modify other weights locally on edges from $u$ and $w$ so that $\phi$ is balanced with respect to the weight. 
     
     Let $\Phi$ be a lift of $\varphi$ to the universal covering. Notice that $\Phi(u)\neq \Phi(w)$, otherwise $\phi(uw)$ is a point, then by Lemma \ref{linedeg}, all the edges incident to $u$ and $w$ are in $\lambda$, contradicting to the assumption. Notice that $\phi(u)$ can be the same point as $\phi(w)$, by $\phi(uw)$ is not degenerate. 

     Consider the angle at the corner $uww_1$ centered at $w$. Since the lengths of $uw$ and $ww_1$ are not zero, and $u,w,w_1$ are in a geodesic, the angle at this corner can only be zero or $\pi$. However, the case of $\pi$ can not occur, otherwise by Lemma \ref{singular}, all the edges incident to $w$ are in $\lambda$. Hence, the angle at the corner $uww_1$ is zero. Similarly, the angle at the corner $wuu_1$ is also zero. 

     After understanding the scenario of one face of the second type, we can conduce the subdivision to all the faces of the second type to produce a new graph $\tilde H$ containing $H$ as a subgraph, and a cell complex $\mathcal{H}$ containing $\mathcal{H}$ as a subcomplex. The cell complex $\tilde{\mathcal{H}}$ contains the $2$-cells in $\tilde{\mathcal{H}}$ and the new  $2$-cells after subdivision such that the images of these two cells are in $\lambda$. Therefore, there is no face of the second type in $\tilde{\mathcal{H}}$, and $\tilde{\mathcal{H}}$ is a union of $2$-cells. 

     In summary, we construct a new cell complex $\tilde{\mathcal{C}}$ and its subcomplex $\tilde{\mathcal{H}}$ such that $\phi(\tilde{\mathcal{H}})$ is in a geodesic $\lambda$. Since $\tilde{\mathcal{H}}$ is not the whole surface, let $v$ be a boundary vertex of $\tilde{\mathcal{H}}$. Then all the corners in $\tilde{\mathcal{H}}$ centered at $w$ are singular, and the angles at these corners can not be $\pi$. Otherwise, by Lemma \ref{singular}, $w$ can not be a boundary vertex of $\tilde{\mathcal{H}}$. Hence, the angle sum at $w$ in  $\tilde{\mathcal{H}}$ is zero. This is true for all boundary vertices of  $\tilde{\mathcal{H}}$.

     Therefore, we sum the angles in $\tilde{\mathcal{H}}$ as the following:
     $$\sum_{\sigma\in\tilde{\mathcal{H}}}(n_\sigma-2)\pi = \sum_{i\in\tilde H}\sum_{i\in\sigma}\theta_\sigma^i = 2\pi|V_I|,$$
     where $V_I$ is the interior vertices of $\tilde{\mathcal{H}}$. On the other hand, since $G$ is essentially $3$-connected, there is digon or monogon, then 
     $$\sum_{\sigma\in\tilde{\mathcal{H}}}(n_\sigma-2)\pi = \pi(f_3+2f_4+3f_5 + 4f_6 + \cdots),$$
     where $f_i$ is the number of $i$-gons in $\tilde{\mathcal{H}}$. Hence, we have 
    $$(f_3+2f_4+3f_5 + 4f_6 + \cdots) = 2|V_I|.$$
    Meanwhile, every interior edge in $E_I$ is used twice and every boundary edge $E_B$ is used once for all the faces, so we have 
    $$3f_3+4f_4+5_5 + 6f_6 + \cdots = 2|E_I|+ |E_B|.$$
    Hence, we have 
    $$2|F|+2|V_I| = 2|E_I| + |E_B|. $$
    However, the Euler characteristics of $\tilde{\mathcal{H}}$, from $|V_B| = |E_B|$, is 
    $$\chi(\tilde{\mathcal{H}}) = |V_I|+|V_B| - |E_I| - |E_B| + |F|  = \frac{1}{2}|E_B|. $$
    Notice that $\chi(\tilde{\mathcal{H}})$ is either non-positive or one. If it is non-positive, then $\tilde{\mathcal{H}}$ is empty, contradicting to the assumption. If it is one, then $\tilde{\mathcal{H}}$ is a disk bounded by two edges and two vertices. This contradicts to the property (3) of the essential 3-connectivity in Theorem \ref{3conn}, where $x, y$ are these two vertices. 
\end{proof}

\end{document}